\documentclass{ws-ijnt}

\usepackage{tikz}

\begin{document}

\markboth{Q. Yang, H. Liu, G. Tang}{A shifted mahler measure identity for Boyd's family}

%
\catchline{}{}{}{}{}
%

\title{A shifted Mahler measure identity for Boyd's family}

\author{Quanli Yang}

\address{ School of Mathematical Sciences, University of Chinese Academy of Sciences\\
	Beijing, 100049, P. R. China, \\
\email{yangquanli16@mails.ucas.ac.cn} }

\author{Hang Liu\footnote{Corresponding author}}

\address{College of Mathematics and Statistics, Shenzhen University\\
	 Shenzhen, Guangdong, 518060, P. R. China\\
liuhang@szu.edu.cn}

\author{Guoping Tang}

\address{ School of Mathematical Sciences, University of Chinese Academy of Sciences\\
	 Beijing, 100049, P. R. China \\
tanggp@ucas.ac.cn }

\maketitle

\begin{abstract}
Recently the second author and Qin numerically verified some Mahler measure identities of genus 2 and 3 polynomial families. In this paper, we use the elliptic regulator to prove an identity invoving shifted Mahler measure for Boyd's family.
\end{abstract}



\section{Introduction}	

The (logarithmic) Mahler measure of a non-zero rational function $P \in \mathbb{C}(x_1,\cdots,x_n)$ is defined by
\begin{align*}
	m(P):=\frac{1}{(2\pi i)^n}\int_{\mathbb{T}^n}\log|P(x_1,\cdots,x_n)|\frac{dx_1}{x_1}\cdots\frac{dx_n}{x_n}
\end{align*}
where the integration is taken over the unit torus $\mathbb{T}^n=\{(x_1,\cdots,x_n)\in \mathbb{C}^n:|x_1|=\cdots=|x_n|=1\}.$

For one-variable polynomials, by the famous Jensen's formula \cite[page 6]{Zudilin2020}, the Mahler measures depend only on the roots of the polynomials. For multivariate polynomials, there are many results that establish the relationship between special value of $L$-functions of arithmetic-geometric objects and these Mahler measures. In his seminal work, Deninger \cite{Deninger1997} related the Mahler measures to regulator integrals and hence found some relationship between the Mahler measures and special values of $L$-functions by means of Beilinson's conjectures.

Boyd \cite{Boyd1998} numerically studied the Mahler measures of families of polynomials like $A(x)y^2+B_k(x)y+C(x)$. For the Boyd's family $Q_k(x,y)=y^2+(x^4+kx^3+2kx^2+kx+1)y+x^4$ and $P_k(x,y)=(x+1)(y+1)(x+y)-kxy$, Bertin and Zudilin \cite{BZ} proved the following relation which was later reproved by Lalin and Wu \cite{Lalin2018} using the regulator theory
\begin{align*}
	m(Q_k)=\begin{cases}
		2m(P_{k}),&\quad 0 \leqslant k \leqslant 4,\\
		m(P_{k}), &\quad k \leqslant -1.
	\end{cases}
\end{align*}
All these families are reciprocal polynomials which are easier to deal with than non-reciprocal polynomials by a relatively standard procedure.

The second author and Qin \cite{Liu} generalized Boyd's method of constructing reciprocal polynomials, obtained more types of polynomials, and proposed many conjectural Mahler measure identities. For $Q_k(x-1,y), P_k(x,y)$ and $R_k(x,y)=x+\frac{1}{x}+y+\frac{1}{y}+(k-4)$, they numerical verified a relation between the Mahler measures of these polynomials.
While we are writing this article, we find Ringeling and Zudilin \cite{Ringeling2021} proved the relation using a diamond-free method. In this article, we prove it in Theorem \ref{theory} by the regulator theory which follows more closely on how the second author and Qin found this relation. We hope the two methods could complement each other.
\begin{theorem}\label{theory}
Let $P_k, Q_k, R_k$ be as above. We have
	\begin{align*}
	m(Q_k(x-1,y))=
	\begin{cases}
	m(R_k)\qquad &k\leqslant -1,\\
	\frac{1}{2}(m(P_k)+m(R_k))\qquad &k\geqslant17.
	\end{cases}
	\end{align*}
\end{theorem}

This paper is organized as follows. Section 2 reviews some important definitions and famous results that we need. In Section 3, we calculate the diamond operators related to these three families of polynomials. In Section 4, we analyze their Deninger paths. In Section 5, we synthesize the results obtained in Sections 3 and 4 and prove Theorem \ref{theory}.

\section{$K_2$ and the regulator theory}

Let $F$ be a field, a Steinberg symbol on $F$ is a bilinear map
\begin{align*}
c:F^*\times F^*\longrightarrow A
\end{align*}
where $A$ is an abelian group whose operation we write multiplicatively for the moment, which satisfies the following condition:
\begin{align*}
c(a,1-a)=1,\quad \text{for all } a \in F^*-\{1\}.
\end{align*}
By Matsumoto's theorem, the second $K$-group of $F$ can be described as
\begin{align*}
K_2(F)\cong F^*\otimes_\mathbb{Z}F^*/\langle a\otimes(1-a):a\in F,a\neq 0,1 \rangle.
\end{align*}
We also call the class $\{a,b\}$ of $a\otimes b$ in $K_2(F)$ the Steinberg symbol.

Let $f,g\in \mathbb{Q}(C)$. We can define a real-meromorphic differential 1-form
\begin{align*}
\eta(\{f,g\}):=\log|f|d\arg g-\log|g|d\arg f,
\end{align*}
where $d\arg f$ is defined by $\text{Im}(df/f)$, $\eta$ is defined outside the zeros and poles of $f$ and $g$.

Beilinson's conjecture relates $K$-theory of varieties to special values of their $L$-functions via the so-called \emph{regulator}. There is a well defined paring between the tame $K_2$ group $K_2^T(C)\subseteq K_2(\mathbb{Q}(C))$ and $H_1(C(\mathbb{C}),\mathbb{Z})$, giving us the \emph{regulator pairing}
\begin{align*}
\langle\cdotp,\cdotp\rangle:H_1(X;\mathbb{Z})\times K_2^T(C)/\text{torsion}\rightarrow \mathbb{R}\\
(\gamma,\alpha)\mapsto \frac{1}{2\pi}\int_{\gamma}\eta(\alpha).
\end{align*}

Let $C$ be the normalization of the projective closure of the algebraic curve defined by $P(x,y)\in\mathbb{C}[x,y]$. Then $\{x,y\}\in K_2^T(C)\otimes \mathbb{Q}$ is equivalent to $P$ being tempered, $i.e.$, the roots of all the face polynomials of $P$ are roots of unity. Denote the degree of $P$ in $y$ by $d$. Write
\begin{align*}
P(x,y)=a_0(x)\prod_{k=1}^{d}(y-y_k(x)),
\end{align*}
where $y_k(x),k=1,\cdots,d$ are $d$ solutions of $P(x,y)=0$ which maybe chosen to be continuous, piecewise analytic functions of $x$. By Jensen's formula with respect to the variable $y$, we have
\begin{align*}
m(P)-m(a_0(x))
&=\frac{1}{(2 \pi i)^2}\int_{\mathbb{T}^2}\sum_{k=1}^{d}\log |y-y_k(x)\frac{dx}{x}\frac{dy}{y}\\
&=\frac{1}{2\pi i}\int_{\mathbb{T}^1}\sum_{k=1}^{d}\log^{+}|y_k(x)|\frac{dx}{x},
\end{align*}
where $\log^{+}|u|:=\max(\log|u|,0).$

In particular, if $d=2$ and $|y_2(x)|\leqslant 1$ as long as $|x|=1$ (this happens if $P$ is reciprocal). Then the above formula can be written as
\begin{align}
\begin{split}\label{eqn:Den}
m(P)-m(a_0(x))
&=\frac{1}{2\pi i}\int_{S}\log|y_1(x)|\frac{dx}{x}\\
&=-\frac{1}{2 \pi}\int_{S}\eta(\{x,y\}),
\end{split}
\end{align}
where $S=\{(x,y):|x|=1,|y_1(x)|\geqslant 1\}.$ When $S$ can be seen as a cycle in $H_1(C,\mathbb{Z})$, then we recover a regulator integral.

The Bloch-Wigner dilogarithm is defined as
\begin{align*}
D(z):=\log|z|\arg(1-z)-\text{Im}\left(\int_{0}^z \log(1-t)\frac{dt}{t}\right),
\end{align*}
where we take the principal branch of the arg function and the path of integration is a straight line segment from 0 to $z$. This is defined on $\{z\in \mathbb{C}: 0<|z|<1\}$, but it can be extended to a real-valued continuous function on $\mathbb{C} \cup \{\infty\}$ which is real-analytic on $\mathbb{C}-\{0,1\}$.

Let $E$ be an elliptic curve over $\mathbb{C}$, choosing $\tau\in H=\{z:\text{Im}z>0\}$ such that $E(\mathbb{C}) \cong \mathbb{C}/(\mathbb{Z} + \mathbb{Z} \tau)$ and set $q=e^{2\pi i\tau}$. A complex point $P$ on $E$ corresponds an element $u+(\mathbb{Z}+\mathbb{Z}\tau)$ of $\mathbb{C}/(\mathbb{Z}+\mathbb{Z}\tau)$. Writing $z=e^{2\pi iu}$, then the \emph{elliptic dilogarithm} is defined as follows
\begin{align*}
D^E(P):=\underset{l\in \mathbb{Z}}{\sum}D(zq^l),
\end{align*}
which we view as a function from $E(\mathbb{C})$ to $\mathbb{R}$.

Let $\mathbb{Z}[E(\mathbb{C})]$ be the group of divisors on $E$ and let
\begin{align*}
\mathbb{Z}[E(\mathbb{C})]^-\cong\mathbb{Z}[E(\mathbb{C})]/\langle(P)+(-P):P\in E(\mathbb{C})\rangle.
\end{align*}
Let $f,g\in \mathbb{C}(E)^\times$, we define a diamond operation by
\begin{align*}
\diamond:\varLambda^2\mathbb{C}(E)^\times\rightarrow\mathbb{Z}[E(\mathbb{C})]^-\\
f \diamond g=\underset{i,j}{\sum}m_i n_j(S_i-T_j),
\end{align*}
where the divisors of $f$ and $g$ are
\begin{align*}
(f)=\underset{i}{\sum}m_i(S_i)\quad \text{and} \quad (g)=\underset{j}{\sum}n_j(T_j).
\end{align*}

\begin{lemma}\label{lemma:diamond}
	The \emph{elliptic dilogarithm} $D^E$ extends by linearity to a map from $\mathbb{Z}[E(\mathbb{Q})]^-$ to $\mathbb{C}$. Let $f,g\in\mathbb{Q}(E)$ and $\{f,g\}\in K_2(E)$.Then \begin{align*}
	r_E(\{f,g\})[\gamma]=D^E((f)\diamond(g)),
	\end{align*}
	where $[\gamma]$ is a generator of $H_1(E,\mathbb{Z})^-$.
\end{lemma}

Let $\sigma$ be an automorphism of order 2 of $C$ and let $f:C\rightarrow C/\langle\sigma\rangle$ be the quotient map. Let $M\in K_2^T(C)$, then we have
\begin{align}
f^*f_{*}(M)&=M\sigma(M), \label{eqn:fbM}\\	
\int_\gamma\eta(f^*f_*(M))&=\int_{f(\gamma)}\eta(f_*(M)) \label{eqn:fbr},
\end{align}
where $f_*$ is the transfer homomorphism and $f^*$ is the restriction homomorphism. $M$ may not be pushed directly to the quotient curve, however there are some ways to push it down to the quotient curve, see Lemma \ref{pro:ab} or Bosman \cite{Bosman2004} for more details.

\section{The diamond operators and the relation between the regulators}

In this section, we first compute the pushforward of $M_1=\{x,y\}$ and $M_2=\{x+1,y\}$ in the $K_2$ of the genus 2 family given by $Q_k(x,y)$ down to the quotient curves of genus 1. Then we compare these pushforwards with the nature elements in $K_2$ of these quotient curves by comparing the diamond operators of these elements.

\subsection{The genus 2 family}

The reciprocal family $Q_k(x,y)=y^2+(x^4+kx^3+2kx^2+kx+1)y+x^4$ generally defines a genus 2 curve $Z_k:Q_k(x,y)=0$, we can make a birational transformation
\begin{equation*}
\begin{cases}
x=\frac{X_1+1}{X_1-1}\qquad\qquad\qquad &X_1=\frac{x_1+1}{x_1-1}   \\
y=\frac{2X_1Y_1-(2k+1)X_1^4+(2k-6)X_1^2-1}{(X_1-1)^4}&Y_1=\frac{4(y^2-x^4)}{y(x-1)^3(x+1)},
\end{cases}
\end{equation*}
where $Z_k'$ is the curve given by $$Y_1^2=(k^2+k)X_1^6+(-2k^2+5k+4)X_1^4+(k^2-5k+8)X_1^2-k+4.$$
Let $h(v)=(k^2+k)v^3+(-2k^2+5k+4)v^2+(k^2-5k+8)v-k+4$.
Then $Z_k'$ is defined by
$Y_1^2=h(X_1^2)$. We can see $Z_k'$ has two automorphisms $\sigma_1:X_1\rightarrow -X_1,Y_1\rightarrow Y_2$ (this corresponds to the automorphism $x\rightarrow \frac{1}{x},y\rightarrow \frac{1}{y}$ of $Z_k$) and $\sigma_2:X_1\rightarrow -X_1,Y_1\rightarrow -Y_1$ (this corresponds to the automorphism $x\rightarrow \frac{1}{x},y\rightarrow \frac{y}{x^4}$ of $Z_k$).

There are two quotient maps
\begin{align*}
f_1:Z_k&\rightarrow Z_k/\langle\sigma_1\rangle:W_0^2=h(Z_0)\\
&(X_1,Y_1)\mapsto(X_1^2,Y_1),
\end{align*}
and
\begin{align*}
f_2:Z_k\rightarrow Z_k/\langle\sigma_2\rangle\\
(X_1,Y_1)\mapsto\left(\frac{1}{X_1^2},\frac{Y_1}{X_1^3}\right),
\end{align*}
where $Z_k/\langle\sigma_2\rangle:W^2=Z^3h\left(\frac{1}{Z}\right)=(4-k)Z^3+(k^2-5k+8)Z^2+(-2k^2+5k+4)Z+k^2+k$.
Making a second transformation
\begin{equation*}
\begin{cases}
(4-K)Z=X\\
(4-K)W=Y,
\end{cases}
\end{equation*}
we have an isomorphism $\psi_1 :Z_k/\langle\sigma_2\rangle\cong E_k$ where $E_k$ is defined by $Y^2=X^3+(k^2-5k+8)X^2+(-2k^2+5k+4)(4-k)X+(k^2+k)(4-k)^2$.
Making another transformation
\begin{equation*}
\begin{cases}
X=4x_1-k\\
Y=8y_1,
\end{cases}
\end{equation*}
we have an isomorphism $\psi_2 :F_k\cong E_k$ where $F_k$ is defined by $y_1^2=x_1\left(x_1^2+\left(\frac{(k-4)^2}{4}-2\right)x_1+1\right)$.
Making the last transformation
\begin{equation*}
\begin{cases}
x_1=\frac{(k-4)+x_0+y_0}{x_0+y_0}\qquad\qquad &x_0=\frac{(k-4)x_1-2y_1}{2x_1(x_1-1)}     \\
y_1=\frac{(k-4)(y_0-x_0)(k-4+x_0+y_0)}{2(x_0+y_0)^2}&y_0=\frac{(k-4)x_1+2y_1}{2x_1(x_1-1)},
\end{cases}
\end{equation*}
we have an isomorphism $\psi_3:R_k\cong F_k$, where $R_k$ represent the loci of  $R_k(x_0,y_0)=x_0+\frac{1}{x_0}+y_0+\frac{1}{y_0}+(k-4).$

There are two useful results from \cite{Liu}

\begin{lemma}
	 $Q_k(x,y)$ and $Q_k(x-1,y)$ are tempered polynomials.
\end{lemma}

\begin{corollary}
	 $\{x+1,y\}\in K_2^T(Z_k)\otimes\mathbb{Q}$.
\end{corollary}

Let us do an analysis as in \cite{Liu}, let $D_{1,k}$ and $D_{2,k}$ be the elliptic curves $Z_k/\langle\sigma_1\rangle$ and $Z_k/\langle\sigma_2\rangle$ respectively. By \eqref{eqn:fbM}, we have
\begin{align*}
f_1^{*}f_{1*}(M_1)=M_1\sigma_1(M_1)=\{x,y\}\left\{\frac{1}{x},\frac{1}{y}\right\}=2M_1,\\
f_2^*f_{2*}(M_2)=M_2\sigma_2(M_2)=\{x+1,y\}\left\{\frac{x+1}{x},\frac{y}{x^4}\right\}=2M_2-M_1.
\end{align*}

Now we can pushforward the regulator integral of $M_2=\{x+1,y\}$ on $Z_k$ to $D_{2,k}$, by above equations and \eqref{eqn:fbr}, we have
\begin{align*}
\frac{1}{\pi}\int_{\gamma}\eta(M_1)=\frac{1}{2\pi}\int_{\gamma}\eta(f_1^{*}f_{1*}(M_1))=\frac{1}{2\pi}\int_{f_1(\gamma)}\eta(f_{1*}(M_1)),
\end{align*}
and
\begin{align*}
\frac{1}{2\pi}\left(2\int_{\gamma}\eta(M_2)-\int_{\gamma}\eta(M_1)\right)=\frac{1}{2\pi}\int_{\gamma}\eta(f_2^{*}f_{2*}(M_2))=\frac{1}{2\pi}\int_{f_2(\gamma)}\eta(f_{2*}(M_2)),
\end{align*}
where $\gamma$ is a cycle in $H_1(Z_k,\mathbb{Z})$. $f_1(\gamma)$ and $f_2(\gamma)$ are cycles in $H_1(D_{1,k},\mathbb{Z})$, $H_1(D_{2,k},\mathbb{Z})$ respectively. Hence, by the above equations, the regulator integral of $M_2$ on $\gamma$ is
\begin{align}
\begin{split}\label{eqn:M12}
\frac{1}{2\pi}\int_{\gamma}\eta(M_2)=\frac{1}{8\pi}\int_{f_1(\gamma)}\eta(f_{1*}(M_1))+\frac{1}{4\pi}\int_{f_2(\gamma)}\eta(f_{2*}(M_2)).
\end{split}
\end{align}

\subsection{The quotient curve $D_{1,k}$}

$D_{1,k}=Z_k/\langle\sigma_1\rangle:W_0^2=(k^2+k)Z_0^3+(-2k^2+5k+4)Z_0^2+(k^2-5k+8)Z_0+(4-k)$, making a birational transformation $\phi_1$
\begin{align*}
\begin{cases}
Z_0=\frac{4X_0+k^2-3k}{k^2+k}\\
W_0=\frac{4(2Y_0+(k-2)X_0+k)}{k^2+k},
\end{cases}
\end{align*}
this gives $\phi_1:D_{1,k}\cong U_k$, where $U_k:Y_0^2+(k-2)X_0Y_0+kY_0=X_0^3$. There is also a birational transformation $\phi_2$
\begin{align*}
\begin{cases}
X_0=k\frac{x_2+y_2+1}{x_2+y_2-k}\qquad\quad &x_2=\frac{X_0-Y_0}{X_0-k}\\
Y_0=k\frac{-kx_2+y_2+1}{x_2+y_2-k} &y_2=\frac{Y_0+(k-1)X_0+k}{X_0-k},
\end{cases}
\end{align*}
which gives $\phi_2:P_k\cong U_k$ where $P_k$ represent the loci of $P_k(x_2,y_2)$.

Let $P$ be the point $(k,k)$ on $U_k$ which is a torsion point of order 6. Lalin and Wu \cite{Lalin2018} get the following result
\begin{lemma}
Let $\gamma_{P_k}=\{(x_2,y_2)\in P_k:|x_2|=1,|y_2|\geqslant 1\}$, $\gamma_{Z_k}=\{(x,y)\in Z_k:|x+1|=1,|y|\geqslant 1\}$ and $[\gamma_2]$ be a generator of $H_1(U_k,\mathbb{Z})^-$. If $[(\phi_1\circ f_1)(\gamma_{Z_k})]=p_1[\gamma_2]$ and $[\phi_2(\gamma_{P_k})]=p_2[\gamma_2]$, then
	\begin{align*}
     r_{Z_k}(\{x,y\})[\gamma_{Z_k}]=p_1D^{U_k}(-6(P)-6(2P)),
	\end{align*}
	and
	\begin{align}
	\begin{split}\label{eqn:rPk}
	r_{P_k}(\{x_2,y_2\})[\gamma_{P_k}]=p_2D^{U_k}(-6(P)-6(2P)),
	\end{split}
	\end{align}
\end{lemma}

By the above Lemma, we have
\begin{align}
\begin{split}\label{eqn:M1}
\int_{f_1(\gamma_{Z_k})}\eta(f_{1*}(M_1))
&=\int_{\gamma_{Z_k}}\eta(f_1^{*}f_{1*}(M_1))\\
&=r_{Z_k}(2\{x,y\})[\gamma_{Z_k}]\\
&=2p_1D^{U_k}(-6(P)-6(2P)).
\end{split}
\end{align}

\subsection{The quotient curve $D_{2,k}$}
First we prove a Lemma which allow us to calculate $f_{2*}(M_2)$ so as to calculate its regulator integral in \eqref{eqn:M12}.
\begin{lemma}\label{pro:ab}
	Suppose we have rational functions $a(Z, W), b(Z, W)$ on $D_{2,k}$ such that
	\begin{align*} a(x+1)+by=1
	\end{align*}
in which we also see $a,b$ as functions on $Z_k$.
	Then we have
	\begin{align*}
	\int_{f_2(\gamma)}\eta(f_{2*}(M_2))=
	&-\int_{f_2(\gamma)}\eta(f_{2*}(\{a,b\}))-\frac{1}{2}\int_{f_2(\gamma)}\eta\left(f_{2*}\left(\left\{a,\frac{y^2}{x^4}\right\}\right)\right)\\
	&-\frac{1}{2}\int_{f_2(\gamma)}\eta\left(f_{2*}\left(\left\{\frac{(x+1)^2}{x},b\right\}\right)\right).	\end{align*}
\end{lemma}
\begin{proof}
	By the properties of Steinberg symbols, we have
	\begin{align}
	\begin{split}\label{eqn:ax}
	0&= \{a(x+1),by\}\\
	&=\{a,b\}+\{a,y\}+\{x+1,b\}+\{x+1,y\}.
	\end{split}
	\end{align}
Now consider the automorphism $\sigma_2: x\rightarrow \frac{1}{x},y\rightarrow \frac{y}{x^4}$ of $Z_k$ acting on \eqref{eqn:ax}, since $a$ and $b$ are invariant under this automorphism, we have
	\begin{align}
	\begin{split}\label{eqn:ab}
	0&= \{a,b\}+\left\{a,\frac{y}{x^4}\right\}+\left\{\frac{x+1}{x},b\right\}+\left\{\frac{x+1}{x},\frac{y}{x^4}\right\}\\
	&=\{a,b\}+\left\{a,\frac{y}{x^4}\right\}+\left\{\frac{x+1}{x},b\right\}+\{x+1,y\}-\{x,y\}.
	\end{split}
	\end{align}
	\begin{align*}
	\eqref{eqn:ax}+\eqref{eqn:ab}&\Rightarrow 0 = 2\{a,b\}+\left\{a,\frac{y^2}{x^4}\right\}+\left\{\frac{(x+1)^2}{x},b\right\}+2\{x+1,y\}-\{x,y\}\\
	&\Rightarrow 2\{x+1,y\}=\{x,y\}-2\{a,b\}-\left\{a,\frac{y^2}{x^4}\right\}-\left\{\frac{(x+1)^2}{x},b\right\}\\
	&\Rightarrow 2 f_2^*f_{2*}(\{x+1,y\})=-2 f_2^*f_{2*}(\{a,b\})-f_2^*f_{2*}\left(\left\{a,\frac{y^2}{x^4}\right\}\right)\\
	&\qquad\qquad\qquad\qquad\qquad-f_2^*f_{2*}\left(\left\{\frac{(x+1)^2}{x},b\right\}\right)\\
	&\Rightarrow \int_{\gamma}\eta (f_2^*f_{2*}(M_2))=-\int_{\gamma}\eta (f_2^*f_{2*}(\{a,b\}))-\frac{1}{2}\int_{\gamma}\eta \left(f_2^*f_{2*}\left(\left\{a,\frac{y^2}{x^4}\right\}\right)\right)\\
	&\qquad\qquad\qquad\qquad-\frac{1}{2}\int_{\gamma}\eta \left(f_2^*f_{2*}\left(\left\{\frac{(x+1)^2}{x},b\right\}\right)\right)\\
	&\Rightarrow \int_{f_2(\gamma)}\eta(f_{2*}(M_2))=-\int_{f_2(\gamma)}\eta(f_{2*}(\{a,b\}))-\frac{1}{2}\int_{f_2(\gamma)}\eta\left(f_{2*}\left(\left\{a,\frac{y^2}{x^4}\right\}\right)\right)\\
	&\qquad\qquad\qquad\qquad\qquad-\frac{1}{2}\int_{f_2(\gamma)}\eta\left(f_{2*}\left(\left\{\frac{(x+1)^2}{x},b\right\}\right)\right)
	\end{align*}
\end{proof}

We can find the following functions satisfying Lemma \ref{pro:ab}
\begin{align*}
	a&=\frac{2(1+Z)}{3+Z}\\
	&=\frac{2(X-k+4)}{X-3k+12},\\
	b&=\frac{(Z-1)^3}{(Z+3)(-Z^2+(2k-6)Z+2W-2k-1)}\\	&=-\frac{(X-4+k)^3}{18(X-3k+12)\left(\frac{1}{2}X^2+(k^2-7k+12)X+k^3-\frac{15}{2}k^2+12k+8+(k-4)Y\right)}.\\	\end{align*}
The functions $\frac{y^2}{x^4}$ and $\frac{(x+1)^2}{x}$ are also invariant under $\sigma_2$, so we can also see them as functions on $D_{2,k}$ and further as functions on $E_k$ which gives
\begin{align*}
\frac{y^2}{x^4}
&=\frac{((2k+1)X_1^4+(-2k+6)X_1^2-2X_1Y_1+1)^2}{(X_1^2-1)^4},\\
&=\frac{(2kZ-Z^2-2k+2W-6Z-1)^2}{(Z-1)^4}\\
&=\frac{4(\frac{1}{2}X^2+(k^2-7k+12)X+(k-4)Y+k^3-\frac{15}{2}k^2+12k+8)^2}{(X+k-4)^2},
\end{align*}
and
\begin{align*}
\frac{(x+1)^2}{x}
&=\frac{4X_1^2}{X_1^2-1},\\
&=\frac{4}{1-Z}\\
&=\frac{4(k-4)}{X+k-4}.
\end{align*}
Let $a_1=X-k+4, a_2=X-3k+12, b_1=X+k-4$ and $b_2=\frac{1}{2}X^2+(k^2-7k+12)X+(k-4)Y+k^3-\frac{15}{2}k^2+12k+8$. Then
\begin{align*}
a&=\frac{2a_1}{a_2}\\
b&=-\frac{b_1^3}{18a_2b_2}\\
\frac{y^2}{x^4}&=\frac{4b_2^2}{b_1^4}\\
\frac{(x+1)^2}{x}&=\frac{4(k-4)}{b_1}.
\end{align*}
Let $T = (k-4,2(k-4)\sqrt{k^2-2k})$, $U=(3k-12,4(k-4)\sqrt{k^2-2k-3})$ and $S=(4-k,16-4k)$ which are all points on $E_k$, and $S$ is a 4-torsion point. Then
\begin{align*}
(a_1)&=(T)+(-T)-2O\\
(a_2)&=(U)+(-U)-2O	\\
(b_1)&=(S)+(-S)-2O\\
(b_2)&=4(S)-4O.
\end{align*}
We see $a, b, \frac{y^2}{x^4}, \frac{(x+1)^2}{x}$ as functions on $Z_k$ or $E_k$ when appropriate. By Lemma \ref{pro:ab}, Lemma \ref{lemma:diamond} and using the fact that $f_{2*}(\{c,d\})=\{c,d\sigma_2(d)\}$ if $c$ is invariant under $\sigma_2$, we have
\begin{align}
\begin{split}\label{eqn:M2}
\int_{f_2(\gamma)}\eta(f_{2*}(M_2))
&=-\int_{f_2(\gamma)}\eta(f_{2*}(\{a,b\}))-\frac{1}{2}\int_{f_2(\gamma)}\eta\left(f_{2*}\left(\left\{a,\frac{y^2}{x^4}\right\}\right)\right)\\
&\quad-\frac{1}{2}\int_{f_2(\gamma)}\eta\left(f_{2*}\left(\left\{\frac{(x+1)^2}{x},b\right\}\right)\right)\\
&=-q_12D^{E_k}((a)\diamond(b))+q_1D^{E_k}\left((a)\diamond\left(\frac{y^2}{x^4}\right)\right)\\
&\quad+q_1D^{E_k}\left(\left(\frac{(x+1)^2}{x}\right)\diamond b \right) \\
&=-2q_1D^{E_k}((a_1)\diamond(b_1))+2q_1D^{E_k}((a_1)\diamond(a_2))+3q_1D^{E_k}((a_2)\diamond(b_1))\\
&\quad-q_1D^{E_k}((b_1)\diamond(b_2))\\
&=-8q_1D^{E_k}((S))-4q_1D^{E_k}((2S))\\
&=-8q_1D^{E_k}((S)),
\end{split}
\end{align}
where $[f_2(\gamma)]=q_1[\gamma']$ in $H_1(E_k,\mathbb{Z})^-$, $[\gamma']$ is a generator of $H_1(E_k,\mathbb{Z})^-$.

Now we take $R_k(x_0,y_0)=0$ into account, then
\begin{align*}
x_0
&=\frac{(k-4)x_1-2y_1}{2x_1(x_1-1)}\\
&=\frac{(2k-8)X+(2k^2-8k)-2Y}{(X+k)(X+k-4)}\\
y_0
&=\frac{(k-4)x_1+2y_1}{2x_1(x_1-1)}\\
&=\frac{(2k-8)X+(2k^2-8k)+2Y}{(X+k)(X+k-4)},
\end{align*}
we can easily get
\begin{align*}
(x_0)\diamond(y_0)\sim -8(S).
\end{align*}
By Lemma \ref{lemma:diamond} again, we have
\begin{align}
\begin{split}\label{eqn:rRk}
r_{R_k}(\{x_0,y_0\})[\gamma_{R_k}]
&=r_{E_k}(\{x_0(X,Y),y_0(X,Y)\})[\psi_2\circ\psi_3(\gamma_{R_k})]\\
&=q_2D^{E_k}((x_0)\diamond(y_0))\\
&=-8q_2D^{E_k}((S)),
\end{split}
\end{align}
where $\gamma_{R_k}=\{(x_0,y_0)\in R_k(x_0,y_0)=0:|x_0|=1,|y_0|\geqslant 1\}$ is the Deninger path of $R_k(x_0,y_0)=0$, and $[\gamma_{R_k}]=q_2[\gamma']$ in $H_1(E_k,\mathbb{Z})^-$ as above. So if we get the multiple relationship between $q_1$ and $q_2$, we can establish the connection between $m(R_k)$ and $\int_{f_2(\gamma)}\eta(f_{2*}(M_2))$.

\section{The cycles of integration}

In the above analysis, if we get the relationship between $p_1$ and $p_2$ and between $q_1$ and $q_2$, we can connect the right side of equations \eqref{eqn:M1} and \eqref{eqn:M2} with $m(P_k)$ and $m(R_k)$ respectively. In this section, we will calculate their relationship in different value ranges of $k$, so as to obtain the Theorem \ref{theory}.

 We will first prove $\gamma_{Z_k}, \gamma_{P_k}$ and $\gamma_{R_k}$ are closed. From now on, let us use \textquotedblleft$\pm$\textquotedblright to indicate the sign to be determined.

For reciprocal polynomials $P_k(x_2,y_2)$ and $R_k(x_0,y_0)$, $\gamma_{P_k}$ and $\gamma_{R_k}$ are closed, since $P_k(x_0,y_0)=0$ and $R_k(x_2,y_2)=0$ do not intersect the torus $\mathbb{T}^2$ for both $k\leqslant -1$ and $k\geqslant 17$.

For $Q_k(x,y)=y^2+(x^4+kx^3+2kx^2+kx+1)y+x^4$, we prove a lemma.
\begin{lemma}
	$\gamma_{Z_k}=\{(x,y)\in Z_k:|x+1|=1, |y|\geqslant 1\}$ is a closed path of $Z_k$ for $k\leqslant-1$ and $k\geqslant 17$.
\end{lemma}	
\begin{proof}
	 Let $Z_k':Q(x-1,y)=0$, then $\gamma_{Z_k}$ being a closed path is equivalent to $\gamma_{Z_k}'=\{(x,y)\in Z_k':|x|=1,|y|\geqslant 1\}$ being a closed path.
	
	 We want to find the intersection between $Q_k(x-1,y)=0$ and the torus $|x|=|y|=1$. Assume such intersection exists, then we also have $Q_k(1/x-1,1/y)=0$. Let $M(x,y)=Q_k(x-1,y)$ and $M^{*}(x,y)=x^4y^2Q_k(\frac{1}{x}-1,\frac{1}{y})$. Then
	 \begin{align*}
	 &Res_y(M(x,y),M^{*}(x,y))=\\
	 &((k+5)x^6+(-k-35)x^5+(-7k+100)x^4+(15k-145)x^3+(-7k+100)x^2+(-k-35)x+k+5)\times\\
	 &(x-1)^2(x^2-x+1)^2x^2k = 0.
	 \end{align*}
	
If $x-1=0$, i.e., $x=1$, then $y=-1$.

If $x^2-x+1=0$, we can find $f(x,y)=y^2 - yxk + yx + x - 1=0$. We have
	 \begin{align*}
	 f(e^{\frac{\pi i}{3}},y)f(e^{-\frac{\pi i}{3}},y)
	 &=y^4+(-k+1)y^3+(k^2-2k)y^2+(-k+1)y+1\\
	 &=y^2g(u),
	 \end{align*}
	 where $u=y+\frac{1}{y}$. It is easy to see $g(u)=u^2+(-k+1)u+k^2-2k-2>0$ for $u\in [-2,2]$ and $k \leqslant -1$ or $k\geqslant 17$.
Hence we can easily conclude that $f(e^{\frac{\pi i}{3}},y)$ and $f(e^{-\frac{\pi i}{3}},y)$ have no solution on $|y|=1$.

Next we prove $(k+5)x^6+(-k-35)x^5+(-7k+100)x^4+(15k-145)x^3+(-7k+100)x^2+(-k-35)x+k+5\neq 0$ for $|x|=1$. Let $s=x+\frac{1}{x}$. Then it is equivalent to show $\mu(k,s)=(s^3-s^2-10s+17)k+5s^3-35s^2+85s-75 \neq 0$ where $s\in [-2,2]$. It is easy to see $s^3-s^2-10s+17>0$. Then we get
	 \begin{align*}
	 \begin{cases}
	 \mu(k,s)\leqslant\mu(-1,s)=4s^3-34s^2+95s-92<0\quad \text{for }k\leqslant-1, \\
	 \mu(k,s)>\mu(17,s)=22s^3-52s^2-85s+214\geqslant 0\quad \text{for }k>17.
	 \end{cases}
	 \end{align*}
	
	 In summary, we see $(1,-1)$ is only intersection point of $\gamma_{Z_k}'$ and $\mathbb{T}^2$ for $k\leqslant-1$ and $k> 17$. Hence $\gamma_{Z_k}'$ is closed for $k\leqslant-1$ and $k> 17$. Then by continuity, $\gamma_{Z_k}'$ is also closed for $k=17$.
\end{proof}

\subsection{The quotient map $f_1$}
First consider the case $k \geqslant 17$. Let $\omega_1$ be the holomorphic differential on $U_k$
\begin{equation*}
\omega_1=\frac{dX_0}{2Y_0+(k-2)X_0+k}.
\end{equation*}
Then we have
\begin{align*}
\int_{f_1(\gamma_{Z_k})}\omega_1
&=\int_{\gamma_{Z_k}}f_1^{*}\omega_1\\
&=\pm\int_{\gamma_{Z_k}}\frac{(x+1)^2}{\sqrt{(x^4+kx^3+2kx^2+kx+1)^2-4x^4}}dx\\
&=\pm\frac{1}{2}\int_{\iota_{1}(\gamma_{Z_k})}\frac{idu}{\sqrt{(1-u)(u+\frac{k-2}{2})(u^2+\frac{k}{2}u+\frac{k}{2})}}.
\end{align*}
where $\iota_1$ is the substitution $u=\frac{1}{2}\left(x+\frac{1}{x}\right)$.

We can parameterize the loop $\gamma_{Z_k}$ by letting $x=e^{2\pi it}-1 (t \in [0,1])$ which gives
\begin{align*}
u(t)
&=\frac{1}{2}\left(x(t)+\frac{1}{x(t)}\right)\\
&=\frac{1}{2}\left(e^{2\pi it}-1+\frac{1}{e^{2\pi it}-1}\right)\\
&=\cos(\pi t)^2-\frac{5}{4}+\frac{3\cos(\pi t)-4\cos(\pi t)^3}{4\sin(\pi t)}i.
\end{align*}
The image of $u(t)$ has only two intersections with the real axis at $-\frac{5}{4}$ which corresponds to $t=\frac{1}{2}$ and $-\frac{1}{2}$ which corresponds to $t=\frac{1}{6}$, $\frac{5}{6}$.

Let $e_1=1$, $e_2=1-\frac{k}{2}$, $e_3=-\frac{1}{4}k+\frac{1}{4}\sqrt{k^2-8k}$ and $e_4=-\frac{1}{4}k-\frac{1}{4}\sqrt{k^2-8k}$. Then for $k\geqslant17$
\begin{equation*}
e_2<e_4<-\frac{5}{4}<e_3<-\frac{1}{2}<1.
\end{equation*}
We can see $u(t)$ is a path looping around $e_3$ which lifts to a homology class equivalent to the lift of the line segment connecting $e_1$ and $e_3$ (see the picture below).
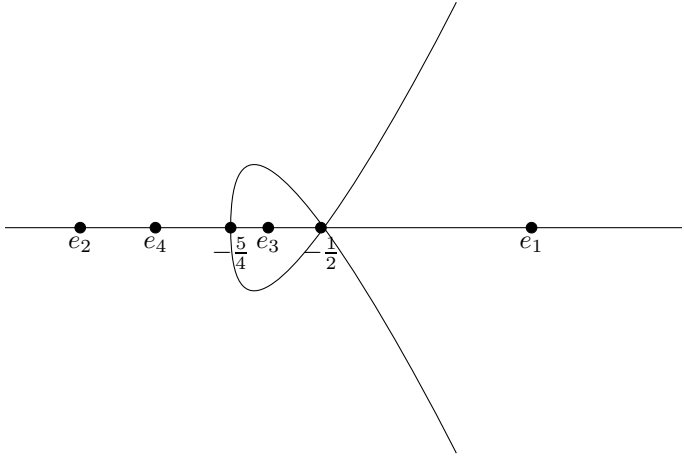
\begin{tikzpicture}
\draw (-1,0)..controls+(down:3cm) and +(right:0cm)..(2,3);
\draw (-1,0)..controls+(up:3cm) and +(right:0cm)..(2,-3);
\draw (-4,0)--(5,0);
\filldraw (-0.5,0) circle (2pt) node[anchor=north]{$e_3$};
\filldraw (-1,0) circle (2pt) node[anchor=north]{$-\frac{5}{4}$};
\filldraw (-2,0) circle (2pt) node[anchor=north]{$e_4$};
\filldraw (-3,0) circle (2pt) node[anchor=north]{$e_2$};
\filldraw (0.2,0) circle (2pt) node[anchor=north]{$-\frac{1}{2}$};
\filldraw (3,0) circle (2pt) node[anchor=north]{$e_1$};
\end{tikzpicture}

Then we get
\begin{align}
\begin{split}\label{eqn:Qkw1}
\int_{f_1(\gamma_{Z_k})}\omega_1
&=\frac{1}{2}\int_{\iota_{1*}(\gamma_{Z_k})}\frac{idu}{\sqrt{(1-u)(u+\frac{k-2}{2})(u^2+\frac{k}{2}u+\frac{k}{2})}}\\
&=\pm\int_{e_3}^{e_1}\frac{idt}{\sqrt{(1-u)(u+\frac{k-2}{2})(u^2+\frac{k}{2}u+\frac{k}{2})}}\\
&=\pm\int_{\frac{k-4+\sqrt{k(k-8)}}{2k}}^{1}\frac{2ids}{\sqrt{s(1-s)(k^2s^2+k(4-k)s+4)}}\quad\left(u=\frac{ks-k+2}{2}\right)\\
&=\pm\int_{\frac{1}{8}k^2-\frac{1}{2}k-1-\frac{1}{8}k\sqrt{k(k-8)}}^{0}\frac{idv}{\sqrt{v(v^2-(\frac{k^2}{4}-k-2)v+k+1)}}\quad\left(s=\frac{1}{v+1}\right).
\end{split}
\end{align}

On the other hand, we have
\begin{align*}
\omega_1
&=\frac{dX_0}{2Y_0+(k-2)X_0+k}\\
&=\frac{y_2dx_2}{(x_2+1)(y_2^2-x_2)},
\end{align*}
and
\begin{align}
\begin{split}\label{eqn:Pkw1}
\int_{\gamma_{P_k}}\omega_1
&=\int_{\gamma_{P_k}}\frac{y_2dx_2}{(x_2+1)(y_2^2-x_2)}\\	
&=\pm\int_{\gamma_{P_k}}\frac{dx_2}{\sqrt{(x_2^4+1)-2k(x_2^3+x_2)+(k^2-4k-2)x_2^2}}\\
&=\pm\int_{0}^{1}\frac{2\pi idt}{\sqrt{(k-4\cos(\pi t)^2)^2-16\cos(\pi t)^2}}\quad\left(x_2=e^{2\pi it}\right)\\
&=\pm\int_{0}^{\frac{1}{2}}\frac{4\pi idt}{\sqrt{(k-4\cos(\pi t)^2)^2-16\cos(\pi t)^2}}\\
&=\pm\int_{0}^{1}\frac{2idu_1}{\sqrt{u_1(1-u_1)((k-4u_1)^2-16u_1)}}\quad\left(u_1=\cos(\pi t)^2\right)\\
&=\pm\int_{0}^{\infty}\frac{idw}{\sqrt{w(w^2+2(\frac{k^2}{4}-k-2)w+\frac{k^3}{16}(k-8))}}\quad\left(u_1=\frac{1}{1+\frac{4w}{k^2}}\right)\\
&=\pm\int_{\frac{1}{8}k^2-\frac{1}{2}k-1+\frac{1}{8}k\sqrt{k(k-8)}}^{\infty}\frac{idv}{\sqrt{v(v^2-(\frac{k^2}{4}-k-2)v+k+1)}}\\
&\qquad\qquad\left(w=v-\left(\frac{k^2}{4}-k-2\right)+\frac{k+1}{v}\right).
\end{split}
\end{align}
So comparing equations \eqref{eqn:Qkw1} and \eqref{eqn:Pkw1} we get
\begin{align*}
\int_{f_1(\gamma_{Z_k})}\omega_1=\pm\int_{\gamma_{P_k}}\omega_1,
\end{align*}
then
\begin{align}\label{eqn:p17}
p_1=\pm p_2
\end{align}
for $k\geqslant17$.

For $k<-1$, we have
\begin{equation*}
-\frac{5}{4}<e_4<-\frac{1}{2}<e_1<e_3<e_2.
\end{equation*}
We can see $u(t)$ as a path looping around $e_4$ twice which lifts to a trivial homology class (see the picture below).

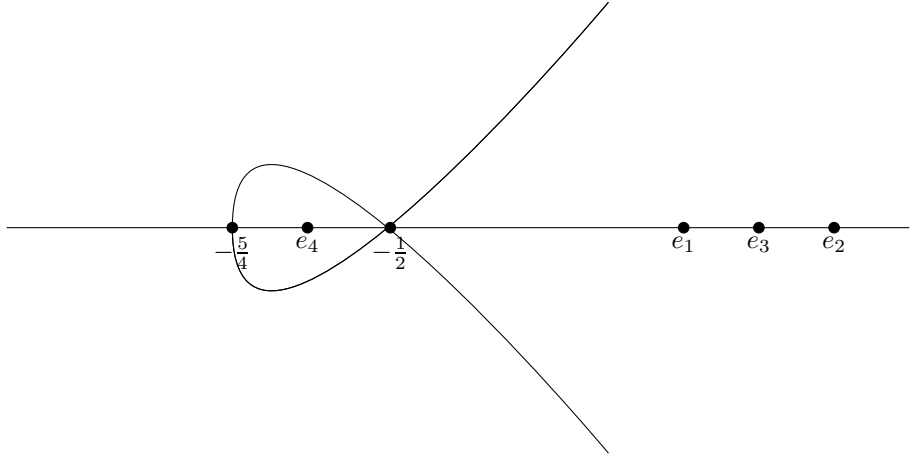
\begin{tikzpicture}
\draw (-6,0)--(6,0);
\draw (-3,0)..controls+(down:3cm) and +(right:0cm)..(2,3);
\draw (-3,0)..controls+(down:3cm) and +(right:0cm)..(2,3);
\filldraw (-3,0) circle (2pt) node[anchor=north]{$-\frac{5}{4}$};
\filldraw (-2,0) circle (2pt) node[anchor=north]{$e_4$};
\filldraw (-0.9,0) circle (2pt) node[anchor=north]{$-\frac{1}{2}$};
\draw (-3,0)..controls+(up:3cm) and +(right:0cm)..(2,-3);
\filldraw (3,0) circle (2pt) node[anchor=north]{$e_1$};
\filldraw (4,0) circle (2pt) node[anchor=north]{$e_3$};
\filldraw (5,0) circle (2pt) node[anchor=north]{$e_2$};
\end{tikzpicture}

We get
\begin{equation*}
\int_{f_1(\gamma_{Z_k})}\omega_1=0.
\end{equation*}
Hence we conclude that $p_1=0$ for $k<-1$. Then by continuity, we get
\begin{equation}\label{p-1}
p_1=0\quad \text{for}\quad k\leqslant-1.
\end{equation}

\subsection{The quotient map $f_2$}

Consider the holomorphic differential
\begin{equation*}
\omega_2=\frac{dX}{2Y} = \frac{dx_1}{4y_1} = \frac{y_0}{2(y_0^2-1)x_0}dx_0,
\end{equation*}
we have
\begin{align}
\begin{split}\label{eqn:Rkw2}
\int_{\gamma_{R_k}}\omega_2
&=\int_{\gamma_{R_k}}\frac{y_0}{2(y_0^2-1)x_0}dx_0\\
&=\pm\int_{\gamma_{R_k}}\frac{dx_0}{2\sqrt{(x_0^2+1+(k-4)x_0)^2-4x_0^2}}\\
&=\pm\int_{0}^{1}\frac{2\pi idt}{2\sqrt{(2\cos(2\pi t)+k-4)^2-4}}\quad(x_0=e^{2\pi i t})\\
&=\pm\pi i\int_{0}^{1}\frac{1}{\sqrt{(4(\cos(\pi t))^2+k-4)(4(\cos(\pi t))^2+k-8)}}dt\\
&=\pm2\pi i\int_{0}^{\frac{1}{2}}\frac{1}{\sqrt{(4(\cos(\pi t))^2+k-4)(4(\cos(\pi t))^2+k-8)}}dt\\
&=\pm2 \int_{0}^{1}\frac{ds}{\sqrt{(s^2-1)(4s^2+k-4)(4s^2+k-8)}}\quad(\cos(\pi t)=s)\\
&=\pm\int_{0}^{1}\frac{du}{\sqrt{u(u-1)(4u+k-4)(4u+k-8)}}\quad(s^2=u)\\
&=\pm\frac{1}{4}\int_{\frac{k}{8}-1}^{\frac{k}{8}}\frac{dw}{\sqrt{(w-\frac{k-8}{8})(w+\frac{k-8}{8})(w-\frac{k}{8})(w+\frac{k}{8})}}\quad\left(u=w-\frac{k}{8}+1\right)\\
&=\pm\int_{-1}^{-\infty}\frac{d\alpha}{\sqrt{(\alpha+1)((k-4)\alpha+4)(k\alpha+4)}}\quad\left(\alpha=\frac{1}{w-\frac{k}{8}}\right)\\
&=\pm\frac{1}{4}\int_{\frac{k}{2}}^{-\infty}\frac{d\beta}{\sqrt{(\beta-\frac{k}{2})(\beta-\frac{k^2}{16})(\beta-\frac{k^2+16}{16})}}\quad\left(\alpha=\frac{16\beta}{k^2-4k}-\frac{k+4}{k-4}\right).
\end{split}
\end{align}
We also have
\begin{align}
\begin{split}\label{eqn:Qkw2}
\int_{f_2(\gamma_{Z_k})}\omega_2
&=\int_{\gamma_{Z_k}}f_2^{*}\omega_2
=\int_{\gamma_{Z_k}}\frac{y(x^2-1)}{2(y^2-x^4)}dx\\
&=\pm\int_{\gamma_{Z_k}}\frac{x^2-1}{2\sqrt{(x^4+kx^3+2kx^2+kx+1)^2-4x^4}}dx \\
&=\pm\frac{1}{4}\int_{\iota_1(\gamma_{Z_k})}\frac{du}{\sqrt{(u+1)(u+\frac{k-2}{2})(u^2+\frac{k}{2}u+\frac{k}{2})}}\\
&=\pm\frac{1}{4}\int_{(\iota_2\circ\iota_1)(\gamma_{Z_k})}\frac{d\tau}{\sqrt{(\tau^2-\frac{(k-4)^2}{16})(\tau^2-\frac{k^2-8k}{16})}}\\
&=\pm\frac{1}{8}\int_{(\iota_3\circ\iota_2\circ\iota_1)(\gamma_{Z_k})}\frac{ds}{\sqrt{(s-\frac{k}{2})(s-\frac{k^2}{16})(s-\frac{k^2+16}{16})}}.
\end{split}
\end{align}
where $\iota_1,\iota_2,\iota_3$ are the substitutions $u=\frac{1}{2}(x+\frac{1}{x})$, $\tau=u+\frac{k}{4}$, $s=\tau^2+\frac{k}{2}$ respectively.

Now let us analyze the integral path $(\iota_3\circ\iota_2\circ\iota_1)(\gamma_{Z_k})$. We parameterize the loop $\gamma_{Z_k}$ by $x=e^{2\pi it}-1$, then
\begin{align*}
s(t)
&=\left(\frac{1}{2}\left(x(t)+\frac{1}{x(t)}\right)+\frac{k}{4}\right)^2+\frac{k}{2}\\
&=-\frac{1}{16}\frac{32\cos(\pi t)^6+(8k-80)\cos(\pi t)^4+(k^2-10k+74)\cos(\pi t)^2-k^2+2k-25}{\sin(\pi t)^2}\\
&\quad -\frac{1}{8}\frac{\cos(\pi t)(4\cos(\pi t)^2-3)(4\cos(\pi t)^2+k-5)}{\sin(\pi t)}i.
\end{align*}
The image of $s(t)$ has only two intersections with the real axis, i.e., $\frac{(k+2)^2}{16}$ which corresponds to $t=\frac{1}{6}$, $\frac{5}{6}$ and $\frac{k^2-2k+25}{16}$ which corresponds to $t=\frac{1}{2}$.
Let $e_1'=\frac{k}{2}$, $e_2'=\frac{k^2}{16}$, $e_3'=\frac{k^2}{16}+1$.

For $k\leqslant-1$, we have
\begin{align*}
e_1'<\frac{(k+2)^2}{16} <e_2'<e_3'<\frac{k^2-2k+25}{16},
\end{align*}
We can see $s(t)$ is a path looping around $e_2'$ and $e_3'$ twice which lifts to a homology class which
is equivalent to twice of the lift of the line segment connecting $e_2'$ and $e_3'$ (see the picture below).

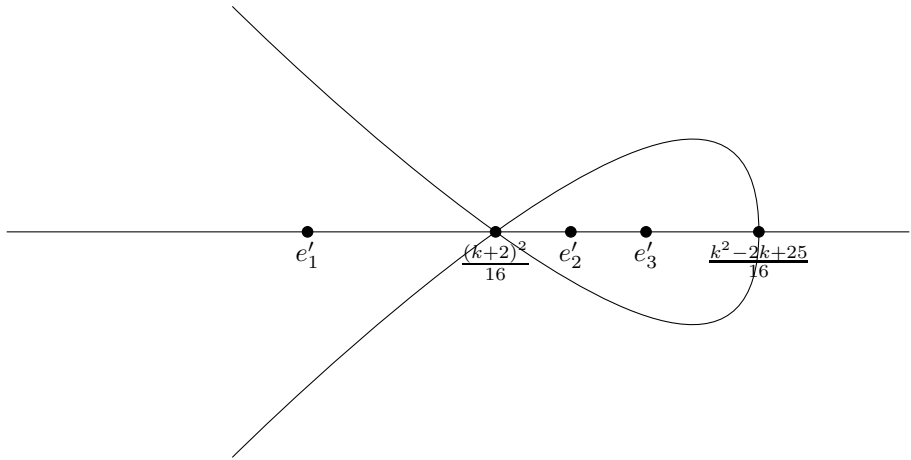
\begin{tikzpicture}
\draw (-6,0)--(6,0);
\draw (4,0)..controls+(down:4cm) and +(right:0cm)..(-3,3);
\draw (4,0)..controls+(up:4cm) and +(right:0cm)..(-3,-3);
\filldraw (4,0) circle (2pt) node[anchor=north]{$\frac{k^2-2k+25}{16}$};
\filldraw (2.5,0) circle (2pt) node[anchor=north]{$e_3'$};
\filldraw (1.5,0) circle (2pt) node[anchor=north]{$e_2'$};
\filldraw (0.5,0) circle (2pt) node[anchor=north]{$\frac{(k+2)^2}{16}$};
\filldraw (-2,0) circle (2pt) node[anchor=north]{$e_1'$};
\end{tikzpicture}

So we have
\begin{align}
\begin{split}\label{eqn:rQk}
&\frac{1}{8}\int_{(\iota_3\circ\iota_2\circ\iota_1)_{*}(\gamma_{Z_k})}\frac{ds}{\sqrt{(s-\frac{k}{2})(s-\frac{k^2}{16})(s-\frac{k^2+16}{16})}}\\
=&\pm\frac{1}{2}\int_{\frac{k}{2}}^{-\infty}\frac{ds}{\sqrt{(s-\frac{k}{2})(s-\frac{k^2}{16})(s-\frac{k^2+16}{16})}}.
\end{split}
\end{align}

Hence by \eqref{eqn:Rkw2}, \eqref{eqn:Qkw2} and \eqref{eqn:rQk}, we conclude that
\begin{align}\label{eqn:q-1}
q_1=\pm2q_2
\end{align}
for $k\leqslant-1$.

For $k\geqslant 17$, we have
\begin{align*}
e_1'< \frac{k^2-2k+25}{16}<e_2'<e_3'<\frac{(k+2)^2}{16}.
\end{align*}
We can see $s(t)$ is a path looping around $e_2'$ and $e_3'$ which lifts to a homology class which
is equivalent to the lift of the line segment connecting $e_2'$ and $e_3'$ (see the picture below).

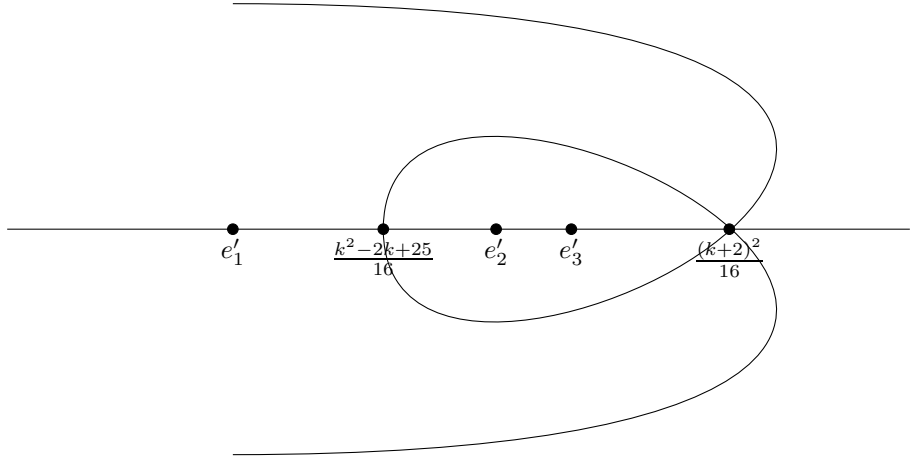
\begin{tikzpicture}
\draw (-1,0)..controls+(down:4cm) and +(right:15cm)..(-3,3);
\draw (-1,0)..controls+(up:4cm) and +(right:15cm)..(-3,-3);
\draw (-6,0)--(6,0);
\filldraw (-1,0) circle (2pt) node[anchor=north]{$\frac{k^2-2k+25}{16}$};
\filldraw (0.5,0) circle (2pt) node[anchor=north]{$e_2'$};
\filldraw (1.5,0) circle (2pt) node[anchor=north]{$e_3'$};
\filldraw (3.6,0) circle (2pt) node[anchor=north]{$\frac{(k+2)^2}{16}$};
\filldraw (-3,0) circle (2pt) node[anchor=north]{$e_1'$};
\end{tikzpicture}

So we have
\begin{align}
\begin{split}\label{eqn:rQk2}
&\frac{1}{8}\int_{(\iota_3\circ\iota_2\circ\iota_1)_{*}(\gamma_{Z_k})}\frac{ds}{\sqrt{(s-\frac{k}{2})(s-\frac{k^2}{16})(s-\frac{k^2+16}{16})}}\\
=&\pm\frac{1}{4}\int_{\frac{k}{2}}^{-\infty}\frac{ds}{\sqrt{(s-\frac{k}{2})(s-\frac{k^2}{16})(s-\frac{k^2+16}{16})}}.
\end{split}
\end{align}
Hence by \eqref{eqn:Rkw2}, \eqref{eqn:Qkw2} and \eqref{eqn:rQk2}, we conclude that
\begin{align}\label{eqn:q17}
q_1=\pm q_2
\end{align}
for $k\geqslant17$.

\section{Proof of Theorem 1.1}

\begin{proof}
Let us take $\gamma=\gamma_{Z_k}$ in \eqref{eqn:M12}. Then we have
	\begin{align*}
	\frac{1}{2\pi}\int_{\gamma}\eta(M_2)
	&=\frac{1}{8\pi}\int_{f_1(\gamma)}\eta(f_{1*}(M_1))+\frac{1}{4\pi}\int_{f_2(\gamma)}\eta(f_{2*}(M_2))\\
	&=\frac{1}{4\pi}p_1D^{U_k}(-6(P)-6(2P))-\frac{2}{\pi}q_1D^{E_k}((S))\qquad(\text{by \eqref{eqn:M2},\eqref{eqn:M1}})\\
	&=\begin{cases}
	\pm\frac{1}{2\pi}r_{\gamma_{R_k}}(\{x_0,y_0\})[\gamma_{R_k}], &k\leqslant-1\\
\pm\frac{1}{4\pi}r_{P_k}(\{x_2,y_2\})[\gamma_{P_k}]\pm\frac{1}{4\pi}r_{\gamma_{R_k}(\{x_0,y_0\})[\gamma_{R_k}]},\quad &k\geqslant 17
\end{cases}
\end{align*}
where the case $k\leqslant-1$ is derived by \eqref{eqn:rPk}, \eqref{eqn:rRk},  \eqref{p-1}, \eqref{eqn:q-1}
and the case $k\geqslant 17$ is derived by
\eqref{eqn:rPk}, \eqref{eqn:rRk}, \eqref{eqn:p17}, \eqref{eqn:q17}.

Hence by \eqref{eqn:Den}, we have
\begin{align}\label{eqn:shift}
m(Q_k(x-1,y))=\left|\frac{1}{2\pi}\int_{\gamma}\eta(M_2)\right|=
	\begin{cases}
	m(R_k),&k\leqslant-1\\
	\frac{1}{2}|m(P_k)\pm m(R_k)|,\quad &k\geqslant 17.
	\end{cases}
	\end{align}
It is easy to see the Mahler measures are equivalent to $\log|k|$ as $k$ tends to infinity. Combine this fact with $\eqref{eqn:shift}$ and continuity, we prove Theorem \ref{theory}.
\end{proof}

By the evaluation of $m(R_{-1}), m(R_{-4}), m(R_{-8})$ and $m(R_{-12})$ in \cite{Br2016}, \cite{La}, \cite{LR}, \cite{RZ} and \cite{RZ14}, we have the following corollary of Theorem \ref{theory}.
\begin{corollary}
Let $E_m/\mathbb{Q}$ be an elliptic curve with conductor $m$. Then we have
\begin{align*}
m(Q_{-1}(x-1,y)) &= 6L'(E_{15},0),\\
m(Q_{-4}(x-1,y)) &= 4L'(E_{24},0),\\
m(Q_{-8}(x-1,y)) &= 2L'(E_{48},0),\\
m(Q_{-12}(x-1,y)) &= 11L'(E_{15},0).
\end{align*}
\end{corollary}

\section{Conclusion}
In this article, we prove a Mahler measure identity involving the shifted Mahler measure of Boyd's family.
We expect the above methods applicable to prove more Mahler measure identities in \cite{Liu}. We will try to develop a universal algorithm to deal with this kind of problem in the future work.

Another possible direction is to deal with the the Mahler measure of polynomials defining curve with non-abelian automorphism groups using this method.

\section{Acknowledgements}
The authors would like to thank Fracois Brunault, Maltilde Lalin, Riccardo Pengo, Haixu Wang and Wadim Zudilin for very helpful conversations and/or correspondences.
The first and third authors were supported by the National Natural Science Foundation of China (Grant No.\,11771422).
The second author was supported by the General Program Class A of Shenzhen Stable Support Plan (Grant No.\,20200812135418001) and the National Natural Science Foundation of China (Grant No.\,11801345).

\end{document}